\documentclass{article}
\usepackage[utf8]{inputenc}
\usepackage{subfig}
\usepackage{graphicx} % Required for inserting images
\usepackage{caption,color}   % 24.62
\usepackage{amssymb}
\usepackage{amsthm}
\usepackage{amsmath}
\usepackage{epic}
\usepackage{setspace}
\usepackage{float}
\usepackage{natbib}
\usepackage{multirow}
\usepackage{hyperref}
\usepackage{xcolor}
\usepackage{marginnote}
\usepackage{amsthm}
\usepackage{authblk}

\newtheorem{thm}{Theorem}[section]

\newtheorem{defi}[thm]{Definition}

\title{\textbf{All eigenvalues of the blowup of a graph}}

\author[1]{\quad Ge Lin\thanks{ linge0717@126.com}}
\author[1]{Changjiang Bu\thanks{ Corresponding author. buchangjiang@hrbeu.edu.cn}}

\affil[1]{School of Mathematical Sciences, Harbin Engineering University, Harbin 150001, PR China }

\date{}

\begin{document}

\maketitle

\begin{abstract}

The $s$-blowup of a graph ($s\geq2$) is the $2s$-uniform hypergraph obtained by replacing each vertex with a set of size $s$ and preserving the adjacency relation.
In this paper,
we define $2s$-weighted graphs and use them to give all eigenvalues of the $s$-blowup of a graph.

\end{abstract}



\section{Introduction}\label{sec1}

The research on spectral hypergraph theory has received extensive attention \cite{bu2026on,clark2021harary,duan2023characteristic,gao2022spectral,lin2025largest,zhang2017spectra}.
There are several classes of uniform hypergraphs whose spectra have been determined.
In 2018,
Clark and Cooper \cite{clark2018hypertrees} obtained all eigenvalues of a uniform hypertree,
and in 2024,
Li et al. \cite{li2024hypertree} gave its characteristic polynomial.
Zheng gave all eigenvalues and the characteristic polynomial of the $3$-uniform complete hypergraph \cite{zheng2020simplifying,zheng2021complete}.
The power hypergraph is the edge expansion hypergraph of a graph.
In 2014,
Zhou et al. \cite{zhou2014some} showed that the eigenvalues of subgraphs of a graph give rise to eigenvalues of the power hypergraph.
In 2023,
Chen et al. further obtained all eigenvalues of the power hypergraph by using signed subgraphs of the graph \cite{chen2023all},
and gave an explicit expression for its characteristic polynomial in 2024 \cite{chen2024spectra}.
This paper considers the vertex expansion hypergraph of a graph,
called its blowup.

The $s$-blowup of a graph ($s\geq2$) is the $2s$-uniform hypergraph obtained by replacing each vertex with a set of size $s$ and preserving the adjacency relation.
In 2015,
Khan and Fan \cite{Fan2015spectral} proved that a graph and its blowup have the same spectral radius.
In 2016,
Khan et al. \cite{khan2016h} gave all $\mathrm{H}$-eigenvalues (i.e., real eigenvalues with real eigenvectors) of the blowup of a graph by using its induced subgraphs.
In 2019,
Fan et al. \cite{fan2019spectral} characterized the spectral symmetry of the blowup of a graph.

Weighted graphs have been widely studied in spectral graph theory.
A signed graph is a weighted graph whose edges are assigned either $1$ or $-1$.
Signed graphs have important applications in the analysis of social networks \cite{facchetti2011computing} and the construction of Ramanujan graphs \cite{marcus2015interlacing}.
An extension of signed graphs is complex unit gain graphs.
A number of spectral properties on complex unit gain graphs have been established \cite{cavaleri2021group,mehatari2022adjacency,reff2012spectral}.
There are some investigations on the specializations of complex unit gain graphs.
These include weighted directed graphs \cite{bapat2012weighted} and Hermitian graphs used to study universal state transfer \cite{cameron2014universal},
as well as the Hermitian adjacency matrix for digraphs and mixed graphs \cite{guo2017hermitian,liu2015hermitian,mohar2020new}.

In this paper,
we define $2s$-weighted graphs in which each vertex is assigned a $2s$-th root of unity,
and the weight of each edge is the product of the weights of its two incident vertices.
We give all eigenvalues of the $s$-blowup of a graph via its $2s$-weighted induced subgraphs (see Theorem \ref{thm1}).

\section{Preliminaries}\label{sec2}

In this section,
we introduce the definition of $2s$-weighted graphs and some notation on uniform hypergraphs.

\begin{defi}

Let $G=(V,E)$ be a graph and let $s$ be a positive integer.
A $2s$-weighted graph $G_{\pi}$ consists of the underlying graph $G$ and $\pi:V\rightarrow\{\xi\in\mathbb{C}:\xi^{2s}=1\}$.
The matrix $A(G_{\pi})=(a_{ij}^{\pi})$ is the adjacency matrix of $G_{\pi}$,
where
\begin{equation*}
a_{ij}^{\pi}=\begin{cases}
\pi(i)\pi(j),&\mbox{if $\{i,j\} \in E$},\\
0,&\mbox{otherwise}.
\end{cases}
\end{equation*}

\end{defi}

When $s=1$,
the $2$-weighted graph $G_{\pi}$ assigns to each vertex a weight of $1$ or $-1$,
which is the marked graph introduced by Harary et al. \cite{harary1977enumeration}.
A marked graph naturally corresponds to a balanced signed graph (i.e., the product of the weights of its edges in every cycle is $1$) if the weight of each edge is the product of the weights of its two incident vertices \cite{harary1981counting}.
Therefore,
$A(G_{\pi})$ is also the adjacency matrix of a balanced signed graph of $G$.

For a $2s$-weighted graph $G_{\pi}$,
the eigenvalues of $A(G_{\pi})$ are called the eigenvalues of $G_{\pi}$.
An induced subgraph of $G_{\pi}$ is called a $2s$-weighted induced subgraph of $G$.

A hypergraph is called \emph{$k$-uniform} if each of its edges contains exactly $k$ vertices.
For a $k$-uniform hypergraph $H=(V(H),E(H))$ with $n$ vertices,
its \emph{adjacency tensor} $A_{H}=(a_{i_{1}i_{2}\cdots i_{k}})$ is a $k$-order $n$-dimensional tensor \cite{cooper2012spectra},
where
\begin{equation*}
a_{i_{1}i_{2}\cdots i_{k}}=\begin{cases}
\frac{1}{(k-1)!},&\text{if $\{i_{1},i_{2},\ldots,i_{k}\}\in E(H)$},\\
0,&\text{otherwise}.
\end{cases}
\end{equation*}
If there exist $\lambda \in \mathbb{C}$ and a nonzero vector $\mathbf{x}=(x_{1},x_{2},\ldots,x_{n})^{\top}\in\mathbb{C}^{n}$ such that
\begin{align}\label{eq1}
\lambda x_{i}^{k-1}=\sum_{i_{2},\ldots,i_{k}=1}^{n}a_{i i_{2} \cdots i_{k}}x_{i_{2}}\cdots x_{i_{k}},\ i=1,\ldots,n,
\end{align}
then $\lambda$ is called an \emph{eigenvalue} of $H$ and $\mathbf{x}$ is an \emph{eigenvector} of $H$ corresponding to $\lambda$ \cite{lim2005singular,qi2005eigenvalues}.

For $U \subseteq V(H)$,
let $\mathbf{x}^{U}=\prod_{u \in U}x_{u}$.
Let $E_{v}(H)$ denote the set of hyperedges of $H$ containing the vertex $v$.
The equation \eqref{eq1} can be written as
\begin{align}\label{eq2}
\lambda x_{v}^{k-1}=\sum_{e \in E_{v}(H)}\mathbf{x}^{e\setminus\{v\}}
\end{align}
for each $v\in V(H)$.

\section{Main results}\label{sec3}

In this section,
we use $2s$-weighted graphs to give all eigenvalues of the $s$-blowup of a graph,
along with some remarks.

For a graph $G=(V,E)$ and $s\geq2$,
let $G^{[s]}$ denote the $s$-blowup of $G$ with the vertex set and hyperedge set
$$V(G^{[s]})=\bigcup_{i \in V}\mathbf{V}_{i}\ \mbox{and}\ E(G^{[s]})=\left\{\mathbf{V}_{i}\cup\mathbf{V}_{j}: \{i,j\} \in E\right\},$$
where each $\mathbf{V}_{i}$ denotes a set with $s$ elements corresponding to the vertex $i$ of $G$,
and all of those sets are pairwise disjoint.
Let $\lambda$ be an eigenvalue of $G^{[s]}$ with corresponding eigenvector $\mathbf{x}$.
By \eqref{eq2},
for each $i \in V$ and each $v \in \mathbf{V}_{i}$,
we have
\begin{align}\label{eq3}\notag
\lambda x_{v}^{2s-1}&=\sum_{j: \{i,j\} \in E}\mathbf{x}^{(\mathbf{V}_{i}\cup\mathbf{V}_{j})\setminus\{v\}}\\
&=\mathbf{x}^{\mathbf{V}_{i}\setminus\{v\}}\sum_{j: \{i,j\} \in E}\mathbf{x}^{\mathbf{V}_{j}}.
\end{align}

We now state the main result as follows.

\begin{thm}\label{thm1}

Let $G$ be a graph and let $s\geq2$.
A complex number $\lambda$ is an eigenvalue of $G^{[s]}$ if and only if
$\lambda$ is an eigenvalue of some $2s$-weighted induced subgraph of $G$.

\end{thm}

\begin{proof}

Let $\lambda$ be an eigenvalue of $G^{[s]}$.
It is known that a $k$-uniform hypergraph ($k\geq3$) has an eigenvalue $0$ \cite{qi2014h}.
On the other hand,
a single vertex is an induced subgraph of $G$ and has an eigenvalue $0$.
Thus, we only need to consider the case of $\lambda\neq0$.

Let $\mathbf{x}$ be an eigenvector of $G^{[s]}$ corresponding to $\lambda$.
In the case of $\lambda\neq0$,
it follows from \eqref{eq3} that the components of $\mathbf{x}$ corresponding to vertices in $\mathbf{V}_{i}$ are either all zero or all nonzero for each $i \in V$.
Then the subhypergraph of $G^{[s]}$ induced by the vertices corresponding to nonzero components of $\mathbf{x}$ is also the $s$-blowup of an induced subgraph of $G$.
We denote this induced subgraph of $G$ by $\widehat{G}=(\widehat{V},\widehat{E})$.
We will prove that there is a $2s$-weighted graph of $\widehat{G}$ with an eigenvalue $\lambda$.

For each $i \in \widehat{V}$,
it follows from \eqref{eq3} that
\begin{align*}
\prod_{v\in\mathbf{V}_{i}}\lambda x_{v}^{2s-1}
&=\prod_{v\in\mathbf{V}_{i}}\left(\mathbf{x}^{\mathbf{V}_{i}\setminus\{v\}}\sum_{j:\{i,j\}\in E}\mathbf{x}^{\mathbf{V}_{j}}\right)\\
&=\prod_{v\in\mathbf{V}_{i}}\left(\mathbf{x}^{\mathbf{V}_{i}\setminus\{v\}}\sum_{j:\{i,j\}\in\widehat{E}}\mathbf{x}^{\mathbf{V}_{j}}\right).
\end{align*}
Then we have
\begin{align*}
\lambda^{s}\left(\mathbf{x}^{\mathbf{V}_{i}}\right)^{2s-1}
=\left(\mathbf{x}^{\mathbf{V}_{i}}\right)^{s-1}\left(\sum_{j:\{i,j\}\in\widehat{E}}\mathbf{x}^{\mathbf{V}_{j}}\right)^{s},
\end{align*}
and so
$$\left(\lambda\mathbf{x}^{\mathbf{V}_{i}}\right)^{s}=\left(\sum_{j: \{i,j\}\in\widehat{E}}\mathbf{x}^{\mathbf{V}_{j}}\right)^{s}.$$
Thus, we get $\lambda \mathbf{x}^{\mathbf{V}_{i}}=\eta_{i}\sum_{j:\{i,j\}\in\widehat{E}}\mathbf{x}^{\mathbf{V}_{j}}$,
where $\eta_{i}$ is an $s$-th root of unity.
For each $i\in\widehat{V}$,
let $\xi_{i}\in\mathbb{C}$ be such that $\xi_{i}^{2}=\eta_{i}$ and note that $\xi_{i}$ is a $2s$-th root of unity.
Denote $y_{i}=\mathbf{x}^{\mathbf{V}_{i}}/\xi_{i}$.
Then we have
$$\lambda y_{i}=\sum_{j:\{i,j\}\in\widehat{E}}\xi_{i}\xi_{j}y_{j}.$$
It implies that there is a $2s$-weighted graph $\widehat{G}_{\pi}$ satisfying $\pi(i)=\xi_{i}$ for all $i \in \widehat{V}$ such that $\lambda$ is an eigenvalue of $\widehat{G}_{\pi}$.

Conversely, let $\lambda\neq0$ be an eigenvalue of some $2s$-weighted induced subgraph $\widehat{G}_{\pi}$ of $G$,
where $\widehat{G}=(\widehat{V},\widehat{E})$ is the underlying graph of $\widehat{G}_{\pi}$.
Next, we prove that $\lambda$ is an eigenvalue of $G^{[s]}$.
Let $\mathbf{y}$ be an eigenvector of $\widehat{G}_{\pi}$ corresponding to $\lambda$.
For each $i \in \widehat{V}$, we have
\begin{align}\label{eq4}
\lambda y_{i}=\sum_{j:\{i,j\}\in\widehat{E}}\pi(i)\pi(j)y_{j}.
\end{align}
Let $z_{i}\in\mathbb{C}$ be such that $z_{i}^{s}=y_{i}$.
Fix some vertex in $\mathbf{V}_{i}$ and denote it by $v_{i}$.
Let $\mathbf{x}=(x_{v})$ be an $s|V|$-dimensional vector with components
\begin{equation*}
x_{v}=\begin{cases}
\pi(i)z_{i}, &\mbox{for $v=v_{i}$ and $i \in \widehat{V}$,}\\
z_{i}, &\mbox{for $v\in\mathbf{V}_{i}\setminus\{v_{i}\}$ and $i \in \widehat{V}$,}\\
0, &\mbox{otherwise.}
\end{cases}
\end{equation*}
We then verify that $\lambda$ and $\mathbf{x}$ satisfy \eqref{eq3}.

For vertices $v \in V(G^{[s]})$ not in the $s$-blowup of the induced subgraph,
the required equation \eqref{eq3} clearly holds.
For each $i \in \widehat{V}$ and $v=v_{i}$,
from $\pi(i)^{2s}=1$ and \eqref{eq4},
we have
\begin{align*}
\lambda x_{v_{i}}^{2s-1}
=\lambda\left(\pi(i)z_{i}\right)^{2s-1}
=\frac{\lambda y_{i} z_{i}^{s-1}}{\pi(i)}
=z_{i}^{s-1}\sum_{j:\{i,j\}\in\widehat{E}}\pi(j)y_{j}.
\end{align*}
Since $\mathbf{x}^{\mathbf{V}_{i}}=\pi(i)z_{i}^{s}=\pi(i)y_{i}$ for each $i \in \widehat{V}$,
we get
\begin{align}\label{eq5}
\lambda x_{v_{i}}^{2s-1}
=\mathbf{x}^{\mathbf{V}_{i}\setminus\{v_{i}\}}\sum_{j:\{i,j\}\in\widehat{E}}\mathbf{x}^{\mathbf{V}_{j}}
=\mathbf{x}^{\mathbf{V}_{i}\setminus\{v_{i}\}}\sum_{j:\{i,j\} \in E}\mathbf{x}^{\mathbf{V}_{j}}.
\end{align}
For each $v\in\mathbf{V}_{i}\setminus\{v_{i}\}$, we similarly get
\begin{align}\label{eq6}
\lambda x_{v}^{2s-1}
=\lambda z_{i}^{2s-1}
=\pi(i)z_{i}^{s-1}\sum_{j:\{i,j\}\in\widehat{E}}\pi(j)y_{j}
=\mathbf{x}^{\mathbf{V}_{i}\setminus\{v\}}\sum_{j:\{i,j\} \in E}\mathbf{x}^{\mathbf{V}_{j}}.
\end{align}
Thus, equations \eqref{eq5} and \eqref{eq6} together verify \eqref{eq3}.
It implies that $\lambda$ is an eigenvalue of $G^{[s]}$.

\end{proof}

\noindent\textbf{Remark 1.}
Let $\widehat{G}$ be an induced subgraph of the graph $G$.
For any $2s$-weighted graph $\widehat{G}_{\pi}$,
denote $|A(\widehat{G}_{\pi})|=(|a_{ij}^{\pi}|)$
and let $\rho(\widehat{G}_{\pi})$ (resp. $\rho(\widehat{G})$) denote the spectral radius of $\widehat{G}_{\pi}$ (resp. $\widehat{G}$).
Since $|A(\widehat{G}_{\pi})| = A(\widehat{G})$, by \citep[Theorem 8.1.18]{Horn2013matrix},
we have $\rho(\widehat{G}_{\pi})\leq\rho(\widehat{G})$.
Note that $\rho(\widehat{G})\leq\rho(G)$.
By Theorem \ref{thm1},
we get that the spectral radius of $G^{[s]}$ is equal to $\rho(G)$.
It implies that our result can derive a result in \cite{Fan2015spectral} and a special case of the result in \cite{yuan2016proof}.
\vspace{4mm}

\noindent\textbf{Remark 2.}
Let $K_{3}$ denote the complete graph with $3$ vertices.
As an illustration of Theorem \ref{thm1},
we give all eigenvalues of the $2$-blowup $K_{3}^{[2]}$.
The induced subgraphs of $K_{3}$ consist of single vertices, single edges, and itself.
The single vertex and its $4$-weighted graph only have an eigenvalue $0$.
It is also easy to check that the eigenvalues of all $4$-weighted graphs of the single edge include $\pm 1$ and $\pm \mathbf{i}$,
where $\mathbf{i}=\sqrt{-1}$.
Let $V=\{1,2,3\}$ denote the vertex set of $K_{3}$.
For any $4$-weighted graph $(K_{3})_{\pi}$ with $\pi:V\rightarrow\{\pm 1, \pm \mathbf{i}\}$,
its characteristic polynomial is
$$\phi_{\pi}(\lambda)=\lambda^{3}-(a_{1}a_{2}+a_{1}a_{3}+a_{2}a_{3})\lambda-2a_{1}a_{2}a_{3},$$
where $a_{i}=\pi(i)^{2}\in\{1,-1\}$ for each $i \in V$.
We consider the following four cases:
\begin{enumerate}

\item If $a_{1}=a_{2}=a_{3}=1$,
then $\phi_{\pi}(\lambda)=\lambda^{3}-3\lambda-2$ with eigenvalues $2$, $-1$, $-1$;

\item If $a_{1}=a_{2}=a_{3}=-1$,
then $\phi_{\pi}(\lambda)=\lambda^{3}-3\lambda+2$ with eigenvalues $-2$, $1$, $1$;

\item If two of $a_{1}, a_{2}, a_{3}$ equal $1$ and the other equals $-1$,
then $\phi_{\pi}(\lambda)=\lambda^{3}+\lambda+2$ with eigenvalues $-1$, $\frac{1\pm\sqrt{7}\mathbf{i}}{2}$;

\item If two of $a_{1}, a_{2}, a_{3}$ equal $-1$ and the other equals $1$,
then $\phi_{\pi}(\lambda)=\lambda^{3}+\lambda-2$ with eigenvalues $1$, $\frac{-1\pm\sqrt{7}\mathbf{i}}{2}$.

\end{enumerate}
By Theorem \ref{thm1},
the eigenvalues of $K_{3}^{[2]}$ are thus
$$\left\{0,\pm 1, \pm \mathbf{i}, \pm 2, \frac{1\pm\sqrt{7}\mathbf{i}}{2}, \frac{-1\pm\sqrt{7}\mathbf{i}}{2}\right\}.$$
\vspace{4mm}

\noindent\textbf{Remark 3.}
A real eigenvalue of a tensor is called an $\mathrm{H}$-eigenvalue if it has a real eigenvector \cite{qi2005eigenvalues}.
%Clearly, $\mathrm{H}$-eigenvalues of a tensor are real.
Qi \cite{qi2005eigenvalues} presented that the real eigenvalues of a tensor are not necessarily $\mathrm{H}$-eigenvalues,
and called those eigenvalues without real eigenvectors $\mathrm{N}$-eigenvalues.
Using the $2$-blowup of the complete graph $K_{3}$,
we will show that there exist uniform hypergraphs having real $\mathrm{N}$-eigenvalues.

It is known from Remark 2 that $-2$ is an eigenvalue of $K_{3}^{[2]}$.
Suppose that $\mathbf{x}$ is an real eigenvector of $K_{3}^{[2]}$ corresponding to $-2$.
For each $i \in V=\{1,2,3\}$,
let $v_{i}$ and $u_{i}$ denote two vertices in $\mathbf{V}_{i}$.
By \eqref{eq3},
we have
\begin{align}\label{eq7}
-2x_{v_{i}}^{3}=x_{u_{i}}\sum_{j \in V\setminus\{i\}}x_{v_{j}}x_{u_{j}} \ \mbox{and}\ -2x_{u_{i}}^{3}=x_{v_{i}}\sum_{j \in V\setminus\{i\}}x_{v_{j}}x_{u_{j}}.
\end{align}
Note that $x_{v_{i}}$ and $x_{u_{i}}$ are either both zero or both nonzero for each $i \in V$.
Then the subhypergraph induced by the vertices corresponding to the nonzero components of $\mathbf{x}$ is the $2$-blowup of some induced subgraph of $K_{3}$.
It also implies that $-2$ is an eigenvalue of the $2$-blowup of this induced subgraph.
It is easy to check that the $2$-blowup of a single vertex and that of a single edge have no eigenvalue $-2$. 
Therefore, the subhypergraph induced by the vertices corresponding to the nonzero components of $\mathbf{x}$ is $K_{3}^{[2]}$,
that is, all components of $\mathbf{x}$ are nonzero.

For each $i \in V$,
dividing the two equations in \eqref{eq7},
we have $x_{v_{i}}^{4}=x_{u_{i}}^{4}$,
and so $x_{v_{i}}=\eta_{i}x_{u_{i}}$,
where $\eta_{i}=1$ or $-1$.
It follows that $-2x_{v_{i}}^{2}/\eta_{i}=\sum_{j \in V\setminus\{i\}}(x_{v_{j}}^{2}/\eta_{j})$ for all $i \in V$.
Denote $y_{i}=x_{v_{i}}^{2}/\eta_{i}$.
We can derive the following system of linear equations
\[
\begin{cases}
2y_1 + y_2 + y_3 = 0, \\
y_1 + 2y_2 + y_3 = 0, \\
y_1 + y_2 + 2y_3 = 0.
\end{cases}
\]
Since the determinant of the coefficient matrix of this system is not equal to $0$,
it has only the trivial solution $y_{1}=y_{2}=y_{3}=0$.
It implies that $\mathbf{x}$ is the zero vector,
which leads to a contradiction.
Thus, there does not exist a real eigenvector of $K_{3}^{[2]}$ corresponding to $-2$,
that is, $-2$ is an $\mathrm{N}$-eigenvalue of $K_{3}^{[2]}$.

\vspace{7mm}
\noindent\textbf{Acknowledgements}
This work is supported by the National Natural Science Foundation of China (No. 12371344),
the Natural Science Foundation for The Excellent Youth Scholars of the Heilongjiang Province (No. YQ2024A009),
and the Fundamental Research Funds for the Central Universities.

\bibliographystyle{plain}
\bibliography{sgraph3}

\begin{thebibliography}{10}

\bibitem{bapat2012weighted}
R.~Bapat, D.~Kalita, and S.~Pati.
\newblock On weighted directed graphs.
\newblock {\em Linear Algebra and its Applications}, 436(1):99--111, 2012.

\bibitem{bu2026on}
C.~Bu, L.~Chen, and Y.~Shi.
\newblock On the second-largest modulus among the eigenvalues of a power
  hypergraph.
\newblock {\em Advances in Applied Mathematics}, 176:103052, 2026.

\bibitem{cameron2014universal}
S.~Cameron, S.~Fehrenbach, L.~Granger, O.~Hennigh, S.~Shrestha, and C.~Tamon.
\newblock Universal state transfer on graphs.
\newblock {\em Linear Algebra and its Applications}, 455:115--142, 2014.

\bibitem{cavaleri2021group}
M.~Cavaleri, D.~D'Angeli, and A.~Donno.
\newblock A group representation approach to balance of gain graphs.
\newblock {\em Journal of Algebraic Combinatorics}, 54(1):265--293, 2021.

\bibitem{chen2023all}
L.~Chen, E.~{van Dam}, and C.~Bu.
\newblock All eigenvalues of the power hypergraph and signed subgraphs of a
  graph.
\newblock {\em Linear Algebra and its Applications}, 676:205--210, 2023.

\bibitem{chen2024spectra}
L.~Chen, E.~{van Dam}, and C.~Bu.
\newblock Spectra of power hypergraphs and signed graphs via parity-closed
  walks.
\newblock {\em Journal of Combinatorial Theory, Series A}, 207:105909, 2024.

\bibitem{clark2018hypertrees}
G.~Clark and J.~Cooper.
\newblock On the adjacency spectra of hypertrees.
\newblock {\em The Electronic Journal of Combinatorics}, 25(2), 2018.
\newblock Article ID \#P2.48.

\bibitem{clark2021harary}
G.~Clark and J.~Cooper.
\newblock A {Harary-Sachs} theorem for hypergraphs.
\newblock {\em Journal of Combinatorial Theory, Series B}, 149:1--15, 2021.

\bibitem{cooper2012spectra}
J.~Cooper and A.~Dutle.
\newblock Spectra of uniform hypergraphs.
\newblock {\em Linear Algebra and its Applications}, 436(9):3268--3292, 2012.

\bibitem{duan2023characteristic}
C.~Duan, E.~{van Dam}, and L.~Wang.
\newblock The characteristic polynomials of uniform double hyperstars and
  uniform hypertriangles.
\newblock {\em Linear Algebra and its Applications}, 678:16--32, 2023.

\bibitem{facchetti2011computing}
G.~Facchetti, G.~Iacono, and C.~Altafini.
\newblock Computing global structural balance in large-scale signed social
  networks.
\newblock {\em Proceedings of the National Academy of Sciences of the United
  States of America}, 108(52):20953--20958, 2011.

\bibitem{fan2019spectral}
Y.~Fan, T.~Huang, Y.~Bao, C.~Zhuan-Sun, and Y.~Li.
\newblock The spectral symmetry of weakly irreducible nonnegative tensors and
  connected hypergraphs.
\newblock {\em Transactions of the American Mathematical Society},
  372(3):2213--2233, 2019.

\bibitem{gao2022spectral}
G.~Gao, A.~Chang, and Y.~Hou.
\newblock Spectral radius on linear $r$-graphs without expanded ${K}_{r+1}$.
\newblock {\em SIAM Journal on Discrete Mathematics}, 36(2):1000--1011, 2022.

\bibitem{guo2017hermitian}
K.~Guo and B.~Mohar.
\newblock Hermitian adjacency matrix of digraphs and mixed graphs.
\newblock {\em Journal of Graph Theory}, 85(1):217--248, 2017.

\bibitem{harary1981counting}
F.~Harary and J.~Kabell.
\newblock Counting balanced signed graphs using marked graphs.
\newblock {\em Proceedings of the Edinburgh Mathematical Society},
  24(2):99--104, 1981.

\bibitem{harary1977enumeration}
F.~Harary, E.~Palmer, R.~Robinson, and A.~Schwenk.
\newblock Enumeration of graphs with signed points and lines.
\newblock {\em Journal of Graph Theory}, 1(4):295--308, 1977.

\bibitem{Horn2013matrix}
R.~Horn and C.~Johnson.
\newblock {\em Matrix {A}nalysis}.
\newblock Cambridge University Press, New York, 2013.

\bibitem{Fan2015spectral}
M.~Khan and Y.~Fan.
\newblock On the spectral radius of a class of non-odd-bipartite even uniform
  hypergraphs.
\newblock {\em Linear Algebra and its Applications}, 480:93--106, 2015.

\bibitem{khan2016h}
M.~Khan, Y.~Fan, and Y.~Tan.
\newblock The {H}-spectra of a class of generalized power hypergraphs.
\newblock {\em Discrete Mathematics}, 339(6):1682--1689, 2016.

\bibitem{li2024hypertree}
H.~Li, L.~Su, and S.~Fallat.
\newblock On a relationship between the characteristic and matching polynomials
  of a uniform hypertree.
\newblock {\em Discrete Mathematics}, 347(5):113915, 2024.

\bibitem{lim2005singular}
L.~Lim.
\newblock Singular values and eigenvalues of tensors: a variational approach.
\newblock In {\em 1st IEEE International Workshop on Computational Advances in
  Multi-Sensor Adaptive Processing}, pages 129--132. IEEE, 2005.

\bibitem{lin2025largest}
H.~Lin, L.~Zheng, and B.~Zhou.
\newblock Largest and least {H}-eigenvalues of symmetric tensors and
  hypergraphs.
\newblock {\em Linear and Multilinear Algebra}, 73(12):2821--2847, 2025.

\bibitem{liu2015hermitian}
J.~Liu and X.~Li.
\newblock Hermitian-adjacency matrices and {H}ermitian energies of mixed
  graphs.
\newblock {\em Linear Algebra and its Applications}, 466:182--207, 2015.

\bibitem{marcus2015interlacing}
A.~Marcus, D.~Spielman, and N.~Srivastava.
\newblock Interlacing families {I}: bipartite {R}amanujan graphs of all
  degrees.
\newblock {\em Annals of Mathematics}, 182(1):307--325, 2015.

\bibitem{mehatari2022adjacency}
R.~Mehatari, M.~Kannan, and A.~Samanta.
\newblock On the adjacency matrix of a complex unit gain graph.
\newblock {\em Linear and Multilinear Algebra}, 70(9):1798--1813, 2022.

\bibitem{mohar2020new}
B.~Mohar.
\newblock A new kind of {H}ermitian matrices for digraphs.
\newblock {\em Linear Algebra and its Applications}, 584:343--352, 2020.

\bibitem{qi2005eigenvalues}
L.~Qi.
\newblock Eigenvalues of a real supersymmetric tensor.
\newblock {\em Journal of Symbolic Computation}, 40(6):1302--1324, 2005.

\bibitem{qi2014h}
L.~Qi.
\newblock {H}$^+$-eigenvalues of {L}aplacian and signless {L}aplacian tensors.
\newblock {\em Communications In Mathematical Sciences}, 12(6):1045--1064,
  2014.

\bibitem{reff2012spectral}
N.~Reff.
\newblock Spectral properties of complex unit gain graphs.
\newblock {\em Linear algebra and its applications}, 436(9):3165--3176, 2012.

\bibitem{yuan2016proof}
X.~Yuan, L.~Qi, and J.~Shao.
\newblock The proof of a conjecture on largest {L}aplacian and signless
  {L}aplacian {H}-eigenvalues of uniform hypergraphs.
\newblock {\em Linear Algebra and its Applications}, 490:18--30, 2016.

\bibitem{zhang2017spectra}
W.~Zhang, L.~Kang, E.~Shan, and Y.~Bai.
\newblock The spectra of uniform hypertrees.
\newblock {\em Linear Algebra and its Applications}, 533:84--94, 2017.

\bibitem{zheng2020simplifying}
Y.~Zheng.
\newblock Simplifying the computation of the spectrum of the complete
  $k$-uniform hypergraph.
\newblock {\em Linear Algebra and its Applications}, 595:145--156, 2020.

\bibitem{zheng2021complete}
Y.~Zheng.
\newblock The characteristic polynomial of the complete 3-uniform hypergraph.
\newblock {\em Linear Algebra and its Applications}, 627:275--286, 2021.

\bibitem{zhou2014some}
J.~Zhou, L.~Sun, W.~Wang, and C.~Bu.
\newblock Some spectral properties of uniform hypergraphs.
\newblock {\em The Electronic Journal of Combinatorics}, 21(4), 2014.
\newblock Article ID \#P4.24.

\end{thebibliography}

\end{document}